\documentclass[12pt]{amsart}
\usepackage[utf8]{inputenc}
\usepackage[T1]{fontenc}
\usepackage[all]{xy}
\usepackage{lipsum}
\usepackage{url}
\usepackage{tikz}
\usepackage{stackrel}
\usepackage{color}

\usepackage{amsmath,amsthm,amssymb}

\newcommand{\aspas}[1]{``{#1}''}

\usepackage{xcolor}

\usepackage{graphicx}

\theoremstyle{definition}

\newtheorem{theorem}{Theorem}[section]
\newtheorem{lemma}[theorem]{Lemma}
\newtheorem{example}[theorem]{Example}

\newtheorem{proposition}[theorem]{Proposition}
\newtheorem{definition}[theorem]{Definition}

\newtheorem{remark}[theorem]{Remark}
\newtheorem{corollary}[theorem]{Corollary}
\numberwithin{equation}{section}

\begin{document}

\renewcommand{\bf}{\bfseries}
\renewcommand{\sc}{\scshape}

\newcommand{\TC}{\text{TC}}

\title[The sectional number of a group homomorphism]%
{The sectional number of a group homomorphism}
\author[C. A. Ipanaque Zapata]{Cesar A. Ipanaque Zapata}
\address[C. A. Ipanaque Zapata]{Departamento de Matem\'atica, IME-Universidade de S\~ao Paulo, Rua do Mat\~ao, 1010 CEP: 05508-090, S\~ao Paulo, SP, Brazil}
\email{cesarzapata@usp.br}

\author[J. Palacios]{Joe Palacios}%
\address[J. Palacios]{Instituto de Matematica y Ciencias Afines, IMCA}
\email{jpalacios@imca.edu.pe}



\begin{abstract} We explore the notion of sectional number of a group homomorphism, leading to a generalization of the covering number of a group, and present several characterizations when the sectional number is finite,  providing estimates for computing this invariant. 
\end{abstract}

\maketitle



\section{Introduction}\label{secintro}
In this article, the term \aspas{homomorphism} refers to a group homomorphism. We write $o(g)$ to refer to the order of the element $g$ and $|G|$ to refer to the order of a group $G$. We also write $\mathrm{Hom}(G,H)$  for the set of all homomorphisms from $G$ into another $H$. 

\medskip Recently, in \cite{zapata2024}, the first author establishes the foundations of sectional number (\cite{schwarz1966}, \cite{berstein1961}) in the categorical framework. This is achieved by introducing a numerical invariant that measures the minimal number of certain local sections of a morphism in a category with a certain topology. 

\medskip In this work, for a homomorphism $f:G\to H$, we study the sectional number of $f$ in the category of groups with the topology given by proper subgroups. This numerical invariant is called the sectional number of $f$, denoted by $\mathrm{sec}(f)$ (see Definition~\ref{defn:etale-sec-number}). In \cite{zapata2024}, the author also introduces the sectional number of $f$ with respect to a quasi Grothendieck topology given by (not necessarily proper) subgroups. This invariant coincides with $\mathrm{sec}(f)$ whenever $f$ does not admit a global section.

\medskip The main results of this paper are as follows: 
\begin{itemize}
    \item We present several equivalent definitions of the sectional number. In particular, we show that the sectional number of an epimorphism coincides with that of the corresponding quotient map (Theorem~\ref{thm:sec-coveringnumberhom}).
    \item We provide a characterization of the finiteness of the sectional number (Theorem~\ref{thm:finiteness-sec}). 
    \item We establish a general upper bound (Theorem~\ref{thm:upper-bound}).
    \item We describe the behavior of the sectional number under weak pullbacks (Theorem~\ref{thm:bajo-weak-pullback}).
    \item We compare the sectional number of a product in terms of the sectional number of its factors (Theorem~\ref{thm:product}).
    \item We provide a general lower bound (Theorem~\ref{thm:lower-bound}).
    \item We give a characterization of the sectional number in terms of maximal elements of its associated poset (Theorem~\ref{thm:sec-poset}).
    \item We show that $\mathrm{sec}(f)$ is the least number of subgroup restrictions needed to locally trivialize the 2-cocycle defining the extension (Theorem~\ref{thm:sec-coho}).
\end{itemize}

This paper is organized into two sections. In Section~\ref{sec:sectional-number}, we review the notion of sectional number along with its basic properties. We introduce new concepts that support and refine this notion. For instance, we define the property of being locally sectionable (Definition~\ref{defn:localsec}), the covering number of a homomorphism (Definition~\ref{defn:covering-number-hom}), and the cyclic covering number (Definition~\ref{defn:cyclic-cov-number}). Section~\ref{sec:applications} presents new insights into group theory, poset theory, and cohomology. For instance, we introduce and study the notion of an $H$-point (Definition~\ref{defn:h-punto}). We also associate a poset to a homomorphism and define the covering number of a poset (Definition~\ref{defn:cov-number-lp}).


\section{Sectional number}\label{sec:sectional-number}
 In this section, we recall the notion of sectional number along with its basic properties.  We adopt the notation used in \cite{zapata2024}. 
 
\subsection{Section and locally sectionable} Consider a homomorphism $f:G\to H$. A global \textit{section} of $f$ is a homomorphism $s:H\to G$ such that $f\circ s=\text{id}_H$. Given a subgroup $L$ of $H$, we say that a homomorphism $s:L\to G$ is a local \textit{section} of $f$ if $f\circ s=\mathrm{incl}_L$, where $\mathrm{incl}_L:L\hookrightarrow H$ is the inclusion homomorphism. 

\medskip Observe that any local section is a monomorphism. In addition, any local section  $s:L\to G$ of $f:G\to H$ is a global section of the restriction homomorphism $f_|:f^{-1}(L)\to L$.

\medskip Furthermore, we have the following remark (see the proof of \cite[Lema 2.5.7, p. 95]{zapata2017}).

\begin{remark}\label{rem:section-semidirectprod}
Let $f:G\to H$ be an epimorphism and $K=\mathrm{Ker}(f)\trianglelefteq G$. \begin{enumerate}
    \item[(1)] Suppose that $f$ admits a global section $s:H\to G$. The section $s$ induces the homomorphism $\widetilde{s}:H\to\mathrm{Aut}(K)$ given by \[\widetilde{s}(b)(a)=s(b)as(b^{-1})\] for any $b\in H$ and $a\in K$. Moreover, there exists an isomorphism $\Omega:K\rtimes_{\widetilde{s}} H\to G$ given by \[\Omega(a,b)=as(b)\] for any $a\in K$ and $b\in H$. Observe that $(f\circ \Omega)(a,b)=b$ for any $(a,b)\in K\rtimes_{\widetilde{s}} H$. Recall that $K\rtimes_{\widetilde{s}} H$ denotes the external semidirect product induced by $\widetilde{s}$ (i.e., the underlying set of this group is to be the set product $K\times H$ and the binary operation on $K\times H$ is defined by the rule $(a_1,b_1)\cdot (a_2,b_2)=(a_1\widetilde{s}(b_1)(a_2),b_1b_2)$, see \cite[p. 75]{robinson2003}). Note that if $G$ is abelian (more generally, if $K\subseteq Z(G)$, where $Z(G)$ denotes the center of $G$), then $\widetilde{s}$ is the trivial homomorphism and  $K\rtimes_{\widetilde{s}} H$ coincides with the usual direct product $K\times H$.
    \item[(2)] Suppose that there exists a homomorphism $\Omega:K\rtimes_{\widetilde{s}} H\to G$ such that \[(f\circ \Omega)(a,b)=b\] for any $(a,b)\in K\rtimes_{\widetilde{s}} H$. The map $s:H\to G$ given by \[s=\Omega\circ \iota_2,\] where $\iota_2:H\to K\rtimes_{\widetilde{s}} H$ is defined by the rule $\iota_2(b)=(1,b)$ for any $b\in H$, is a global section of $f$.  
\end{enumerate}  
\end{remark}

Hence, by Remark~\ref{rem:section-semidirectprod}, an epimorphism $f:G\to H$ admits a global section (i.e., the short exact sequence $1\to\mathrm{Ker}(f)\hookrightarrow G\stackrel{f}{\to}H\to 1$ splits) if and only if there exists an isomorphism $\Omega:K\rtimes H\to G$ such that $f\circ\Omega=\pi_2$.


\medskip We introduce the following notion.

\begin{definition}\label{defn:localsec}
 A homomorphism $f:G\to H$ is \textit{locally sectionable} if for each element $b\in H$ with $b\neq 1$, there exists a subgroup $L$ of $H$ such that $b\in L$ and $f$ admits a local section defined in $L$, equivalently, there exists a collection $\{H_\lambda\}_{\lambda\in\Lambda}$ of subgroups of $H$ such that $H=\bigcup_{\lambda\in\Lambda}H_\lambda$ and $f$ admits a local section defined in each $H_\lambda$.    
\end{definition} 

Let $G$ and $H$ be groups. Observe that if there exists a locally sectionable homomorphism from $G$ to $H$, then there exists a collection $\{H_\lambda\}_{\lambda\in\Lambda}$ of subgroups of $H$ such that $H=\bigcup_{\lambda\in\Lambda}H_\lambda$ and for each $\lambda\in\Lambda$, $G$ admits a copy of $H_\lambda$ as a subgroup. In particular, for each element $x\in H$ there exists an element $y\in G$ such that $o(y)=o(x)$. 

\medskip In addition, we have the following remark. 

\begin{remark}\label{rem:locally-sec-implies-epi}
Let $f:G\to H$ and $g:H\to K$ be homomorphisms.
\begin{enumerate}
    \item[(1)] If $f$ is locally sectionable (for example if it admits a global section), then $f$ is an epimorphism. The other implication does not hold, for instance, the quotient map $\mathbb{Z}\to\mathbb{Z}_n:=\mathbb{Z}/n\mathbb{Z}$ with $n\neq 0$, is an epimorphiam but is not locally sectionable  (indeed, observe that, for any subgroup $L$ of $\mathbb{Z}/n\mathbb{Z}$, there is not a nontrivial homomorphism from $L$ to $\mathbb{Z}$ because $L$ is a torsion group and $\mathbb{Z}$ is torsion free).
    \item[(2)] If $f$ and $g$ are locally sectionable, then $g\circ f$ is also locally sectionable. 
\end{enumerate} 
\end{remark}

The next statement shows that the other implication of Remark~\ref{rem:locally-sec-implies-epi}(1) holds whenever $H$ is torsion free.

\begin{lemma}\label{lem:torsion-free-localsec-equiv-epi}
    Let $f:G\to H$ be a homomorphism with $H$ nontrivial. 
   \begin{enumerate}
       \item[(1)] Assume that $H$ is torsion free. We have $f$ is locally sectionable if and only if $f$ is an epimorphism.
       \item[(2)] Assume that $H$ is a torsion group. We have $f$ is locally sectionable if and only if for each $b\in H$ there exists an element $a\in G$ such that $f(a)=b$ and $o(a)=o(b)$.
   \end{enumerate} 
\end{lemma}
\begin{proof}
\noindent \begin{enumerate}
       \item[(1)] The sufficiency of the condition follows from the Definition~\ref{defn:localsec}. Now, suppose that $f$ is an epimorphism. Let $b\in H$ with $b\neq 1$. Since $f$ is surjective, there exists an element $a\in G$ such that $f(a)=b$. Consider the map $s:\langle b\rangle\to G$ given by \[s(b^j)=a^j\] for any $j\in\mathbb{Z}$. Since $b$ is of order infinite (here we use that $H$ is torsion free), $s$ is well defined. Observe that $s$ is a homomorphism. In addition, $(f\circ s)(b^j)=b^j$ for any $j\in\mathbb{Z}$. Consider $L=\langle b\rangle$. Hence, $s:L\to G$ is a local section of $f$ with $b\in L$. Therefore, $f$ is locally sectionable. 
  \item[(2)] Suppose for each $b\in H$ there exists an element $a\in G$ such that $f(a)=b$ and $o(a)=o(b)$. Let $b\in H$ be an element with $b\neq 1$. Consider $a\in G$ such that $f(a)=b$ and $o(a)=o(b)$. Let $s:\langle b\rangle\to G$ be the map given by \[s(b^j)=a^j\] for any $j\in\mathbb{Z}$. Since $o(a)=o(b)$, $s$ is well defined. Likewise, as in Item (1), consider $L=\langle b\rangle$, and we conclude that $s:L\to G$ is a local section of $f$ with $b\in L$. Therefore, $f$ is locally sectionable.  

  The other implication is given by Definition~\ref{defn:localsec}. 
  \end{enumerate} 
\end{proof}

 \subsection{Sectional number} From \cite[p. 492]{harver1959} (see also \cite[p. 1071]{rosenfeld1963}, \cite[p. 44]{cohn1994}), given a group $G$, the \textit{covering number} of $G$, denoted by $\sigma(G)$, is the least positive integer $m$ such that there are $m$ distinct proper subgroups $G_j$ of $G$ with $G=G_1\cup\cdots\cup G_m$. Observe that if $G$ is Noetherian, then we can assume that each $G_i$ is maximal in $G$, so has prime index. We set $\sigma(G)=\infty$ if no such $m$ exists. For instance, $\sigma(G)=\infty$ whenever $G$ is cyclic. Also, we have $\sigma(\mathbb{Q})=\infty$ (see after of \cite[Example 3.12]{zapata2024}). If $f:G\to H$ is an epimorphism, then $\sigma(G)\leq\sigma(H)$. In particular, if $G$ is a group and $N\trianglelefteq G$ is a normal subgroup, then $\sigma(G)\leq \sigma(G/N)$. It is well known that $\sigma(G)\geq 3$ for any group $G$ and the equality holds if and only if there exists an epimorphism $G\to \mathbb{Z}_2\times\mathbb{Z}_2$. 
 
 \medskip In \cite[Example 3.12(i)]{zapata2024}, the first author introduced the notion of sectional number of a  morphism in an abstract category.

\begin{definition}[Sectional Number]\label{defn:etale-sec-number}
Let $f:G\to H$ be a homomorphism. The \textit{sectional number} of $f$, denoted by $\mathrm{sec}(f)$, is the least positive integer $m$ such that there exist proper subgroups $H_1,\ldots,H_m$ of $H$ such that $H=H_1\cup\cdots\cup H_m$, and for each $i=1,\ldots,m$, there exists a local section  $\sigma_i:H_i\to G$ of $f$, i.e., a homomorphism satisfying $f\circ\sigma_i=\text{incl}_{H_i}$, that is, the following diagram commutes:
\begin{eqnarray*}
\xymatrix{ G \ar[rr]^{f} & &H  \\
        &  H_i\ar@{-->}[lu]^{\sigma_i}\ar@{^{(}->}[ru]_{\text{incl}_{H_i}} & } 
\end{eqnarray*}  We set $\mathrm{sec}(f)=\infty$ if no such $m$ exists.
\end{definition}

Since the local section condition is an extra constraint, it follows that $\mathrm{sec}(f)\geq \sigma(H)$ for any homomorphism $f:G\to H$. Furthermore, if $\mathrm{sec}(f)<\infty$, then $f$ is locally sectionable and thus $f$ is an epimorphism. The other implication is not true. For instance, any epimorphism $f:G\to \mathbb{Q}$ whose codomain is the additive group of rational numbers, is locally sectionable, see Lemma~\ref{lem:torsion-free-localsec-equiv-epi}(1); however, $\text{sec}(f)=\infty$ because $\infty=\sigma(\mathbb{Q})\leq\mathrm{sec}(f)$. 

\medskip We have the following example.

\begin{example}\label{exam:codomain-cyclic}
\noindent\begin{enumerate}
\item[(1)] We have that $\mathrm{sec}(\overline{1})=\infty$, where $\overline{1}:G\to H$ is the trivial homomorphism (i.e., $\overline{1}(g)=1$ for any $g\in G$). 
    \item[(2)] Let $f:G\to H$ be a homomorphism. If $G$ or $H$ is a cyclic group, then \[\mathrm{sec}(f)=\infty.\] Note that the equality $\mathrm{sec}(f)=\infty$ always holds whenever $f$ is not surjective. For the case $f$ is surjective, recall the inequalities $\mathrm{sec}(f)\geq\sigma(H)\geq\sigma(G)$.
    \item[(3)] Let $f:G\to H$ be a homomorphism where $G$ is torsion free and $H$ is a torsion group. We have \[\mathrm{sec}(f)=\infty\] because $f$ is not locally sectionable. Indeed, for any subgroup $L$ of $H$, there is not a nontrivial homomorphism from $L$ to $G$ because $L$ is a torsion group and $G$ is torsion free.
\end{enumerate}  
\end{example}

The following remark shows that the sectional number is a generalization of the covering number.

\begin{remark}\label{rem:basic-rem}
 Let $f:G\to H$ be a homomorphism. If $f$ admits a global section, then \[\mathrm{sec}(f)=\sigma(H)\geq \sigma(G).\] In particular, $\mathrm{sec}(f)=\sigma(H)=\sigma(G)$ whenever $f:G\to H$ is an isomorphism.  
\end{remark} 

\begin{example}\label{exam:epi-vector-spaces}
\noindent\begin{enumerate}
    \item[(1)] Let $G$ and $H$ be vector spaces over a field $\mathbb{K}$. Let $f:G\to H$ be an $\mathbb{K}$-epimorphism. Since any vector space has a basis, $f$ admits a global section and $\mathrm{sec}(f)=\sigma(H)$. 
    \item[(2)] Let $G$ be a noncyclic \textit{elementary abelian $p$-group} for some prime $p$, i.e., $G$ is abelian and all its non-identity elements have the same order $p$ (and, of course, $G$ is a vector space over $\mathbb{F}_p$ equipped with the $\mathbb{F}_p$-module structure given by  $(z+p\mathbb{Z})a:=za$ for any $z+p\mathbb{Z}\in \mathbb{F}_p:=\mathbb{Z}/p\mathbb{Z}$ and $a\in G$). Assume that $f:G\to \mathbb{Z}_p^n$ is an epimorphism, where $\mathbb{Z}_p^n:=\mathbb{Z}_p\times\cdots\times \mathbb{Z}_p$ ($n$ times), with $n\geq 2$. Observe that $f$ is indeed a $\mathbb{F}_p$-epimorphism, and hence, by Item (1), \[\mathrm{sec}(f)= p+1.\] Recall that $\sigma(\mathbb{Z}_p^n)=p+1$ with $n\geq 2$ (see \cite[p. 1071]{rosenfeld1963}).  
\end{enumerate}   
\end{example}

We will see that the sectional number of $f:G\to H$ can be significantly larger than the covering number of $H$. In view of Remark~\ref{rem:basic-rem}, such a map $f$ cannot admit a global section.

\medskip First, we present the following statement.

\begin{proposition}\label{prop:diagram}
  Consider the commutative diagram of homomorphisms:  
\begin{eqnarray*}
\xymatrix{ G^\prime \ar[d]^{f^\prime} & G \ar[r]^{\varphi} \ar[d]_{f} & G^\prime \ar[d]^{f^\prime} & \\
     H^\prime \ar[r]_{\xi} & H  \ar[r]_{\psi} &  H' &}
\end{eqnarray*} Assume that $\psi\circ \xi=\mathrm{id}_{H^\prime}$. If $\xi$ is an epimorphism or $f'$ does not admit a global section, then \[\mathrm{sec}(f)\geq\mathrm{sec}(f^\prime)\]
  and the equality holds whenever $\varphi$ is an isomorphism and $\xi$ is an epimorphism. 
 \end{proposition}
 \begin{proof}
 Let $L$ be a subgroup of $H$ and consider $L':=\xi^{-1}(L)$, it is a subgroup of $H'$. Note that a homomorphism $s:L\to G$ produces a homomorphism  $\delta=\left(L'\stackrel{\xi}{\to} L\stackrel{s}{\to} G\stackrel{\varphi}{\to} G^\prime\right).$ If $\psi\circ \xi=\mathrm{id}_{H^\prime}$ and $f\circ s=\mathrm{incl}_L$, then $f^\prime\circ\delta=\mathrm{incl}_{L'}$. 

Now, if $\mathrm{sec}(f)=m$ and $\{H_1,\ldots,H_m\}$ is a collection of proper subgroups of $H$ such that $H=H_1\cup\cdots\cup H_m$ and in each $H_i$ there exists a local section $s_i:H_i\to G$ of $f$. By the construction above, in each subgroup $\xi^{-1}(H_i)$ there exists a local section $\delta_i:\xi^{-1}(H_i)\to G'$ of $f'$. Since $\xi$ is an epimorphism or $f'$ does not admit a global section, $\{\xi^{-1}(H_1),\ldots,\xi^{-1}(H_m)\}$ is a collection of proper subgroups of $H'$ satisfying $H'=\xi^{-1}(H_1)\cup\cdots\cup\xi^{-1}(H_m)$. Hence, we conclude that $\mathrm{sec}(f')\leq m=\mathrm{sec}(f)$.  
  
Now, suppose that $\varphi$ is an isomorphism. Notice that the other inequality $\mathrm{sec}(f)\leq\mathrm{sec}(f^\prime)$ follows from the commutative diagram:
\begin{eqnarray*}
\xymatrix{ G \ar[d]^{f} & G^\prime \ar[r]^{\varphi^{-1}} \ar[d]_{f^\prime} & G \ar[d]^{f} & \\
     H \ar[r]_{\psi} & H^\prime  \ar[r]_{\xi} &  H &}
\end{eqnarray*}
 \end{proof} 

Motivated by the notion of covering number, we present the following definition.

\begin{definition}[Covering Number of a Homomorphism]\label{defn:covering-number-hom}
   Let $f:G\to H$ be a homomorphism. The \textit{covering number} of $f$, denoted by $\sigma(f)$, is the least positive integer $m$ such that there exist $G_1,\ldots,G_m$ proper subgroups of $G$ such that $G=G_1\cup\cdots\cup G_m$ and for each $i=1,\ldots,m$, $\mathrm{Ker}(f)\subsetneq G_i$ and there exists an isomorphism $\Omega_i:\mathrm{Ker}(f)\rtimes f(G_i)\to G_i$ such that $f\circ\Omega_i=\pi_2$. 
\end{definition}

Now we present some equivalent definitions of sectional number. In particular, we show that the sectional number of an epimorphism coincides with the sectional number of a quotient map.

\begin{theorem}\label{thm:sec-coveringnumberhom}
    Let $f:G\to H$ be an epimorphism.  
  \begin{enumerate}
        \item[(1)] We have \[\mathrm{sec}(f)\geq \sigma(f)\] and the equality holds whenever $f$ does not admit a global section. 
        \item[(2)] We have \[\mathrm{sec}(f)=\mathrm{sec}(q_f),\] where $q_f:G\to G/\mathrm{Ker}(f)$ is the usual quotient map (i.e., $q_f(g)=g\mathrm{Ker}(f)$ for all $g\in G$).
        \item[(3)] Suppose that $f$ does not admit a global section. We have that $\mathrm{sec}(f)=\sigma(q_f)$.
    \end{enumerate}
\end{theorem}
\begin{proof}
 \noindent\begin{enumerate}
        \item[(1)] Suppose that $\mathrm{sec}(f)=m<\infty$. Consider proper subgroups $H_1,\ldots,H_m$ of $H$ such that $H=H_1\cup\cdots\cup H_m$, and for each $i=1,\ldots,m$, the restriction homomorphism $f_|:f^{-1}(H_i)\to H_i$ admits a global section. For each $i=1,\ldots,m$, let $G_i=f^{-1}(H_i)$, it is a proper subgroup of $G$ (here we use that $f$ is an epimorphism). Note that $G=G_1\cup\ldots\cup G_m$. Observe that $f(G_i)=H_i$ (here we use again that $f$ is an epimorphism), $\mathrm{Ker}(f)\subsetneq G_i$ (here we use $H_i\neq\{1\}$), and $f_|=f_{|G_i}$. Then, by Remark~\ref{rem:section-semidirectprod}(1), there exists an isomorphism $\Omega_i:\mathrm{Ker}(f)\rtimes f(G_i)\to G_i$ such that $f\circ\Omega_i=\pi_2$ for each $i=1,\ldots,m$. Hence, $\sigma(f)\leq m=\mathrm{sec}(f)$.

    Now, suppose that $f$ does not admit a global section. Let $\sigma(f)=m<\infty$ and consider proper subgroups $G_1,\ldots,G_m$ of $G$ such that $G=G_1\cup\cdots\cup G_m$ and for each $i=1,\ldots,m$, $\mathrm{Ker}(f)\subsetneq G_i$ and there exists an isomorphism $\Omega_i:\mathrm{Ker}(f)\rtimes f(G_i)\to G_i$ such that $f\circ\Omega_i=\pi_2$. For each $i=1,\ldots,m$, let $H_i=f(G_i)$, it is a proper subgroup of $H$ (here we use that $f$ does not admit a global section). Note that $H=H_1\cup\cdots\cup H_m$ (here we use that $f$ is an epimorphism). By Remark~\ref{rem:section-semidirectprod}(2), there exists a global section $s_i:H_i\to G_i$ of $f_{|G_i}:G_i\to H_i$, and, of course, $s_i$ is a global section of $f_|:f^{-1}(H_i)\to H_i$ (observe that $G_i\subseteq f^{-1}(H_i)$). Hence, $\mathrm{sec}(f)\leq m=\sigma(f)$. Therefore, $\mathrm{sec}(f)=\sigma(f)$. 
    \item[(2)] It follows from Proposition~\ref{prop:diagram} applying to the commutative diagram:
    \begin{eqnarray*}
\xymatrix{ G \ar[d]^{q_f} & & G \ar[rr]^{\mathrm{id}_G} \ar[d]_{f} & & G \ar[d]^{q_f}  \\
     G/\mathrm{Ker}(f) \ar[rr]_{\overline{f}} & & H  \ar[rr]_{(\overline{f})^{-1}} & &  G/\mathrm{Ker}(f) }
\end{eqnarray*} where $\overline{f}:G/\mathrm{Ker}(f)\to H$, given by $\overline{f}\left(g\mathrm{Ker}(f)\right)=f(g)$, is an isomorphism by the Fundamental Homomorphism Theorem.
\item[(3)] It follows from Item (1) and (2). Observe that $f$ does not admit a global section if and only if $q_f$ does not admit a global section.
    \end{enumerate}
\end{proof} 

 We have the following example.

\begin{example}\label{exam:covering-number-quotient}
 Let $G$ be a group and $N\trianglelefteq G$ be a normal subgroup. Let $q:G\to G/N$ be the usual quotient map. We know that the inequalities $\mathrm{sec}(q)\geq \sigma(G/N)\geq\sigma(G)$ always hold. Furthermore:
 \begin{enumerate}
     \item[(1)] If $q$ admits a global section (equivalently, there exists an isomorphism $\Omega:N\rtimes G/N\to G$ such that $q\circ\Omega=\pi_2$), by Remark~\ref{rem:basic-rem}, then \[\mathrm{sec}(q)=\sigma(G/N).\] 
      \item[(2)] If $q$ does not admit a global section, by Theorem~\ref{thm:sec-coveringnumberhom}(3), then $\mathrm{sec}(q)$ coincides with the least positive integer $m$ such that there exist $G_1,\ldots,G_m$ proper subgroups of $G$ such that $G=G_1\cup\cdots\cup G_m$ and for each $i=1,\ldots,m$, $N\subsetneq G_i$ and there exists an isomorphism $\Omega_i:N\rtimes G_i/N\to G_i$ such that $q\circ\Omega_i=\pi_2$.
 \end{enumerate}
\end{example}

Theorem~\ref{thm:sec-coveringnumberhom}(2) implies the following result.

\begin{corollary}
    Let $f:G\to H$ and $g:G\to K$ be homomorphisms. If $\mathrm{Ker}(f)=\mathrm{Ker}(g)$, then \[\mathrm{sec}(f)=\mathrm{sec}(g).\]
\end{corollary}

On the other hand, we present a characterization of the finiteness of the sectional number. 

\begin{theorem}[Finiteness]\label{thm:finiteness-sec}
Let $f:G\to H$ be an epimorphism. If there exists an epimorphism $\varphi:H\to K$ such that $K$ is a finite noncyclic group and $f$ admits a local section on $\varphi^{-1}(\langle b\rangle)$ for each $b\in K$ with $b\neq 1$, then \[\mathrm{sec}(f)<\infty.\] The other implication holds whenever $H$ is abelian.  
\end{theorem}
\begin{proof}
 Since $K=\bigcup_{\overset{b\in K}{b\neq 1}}\langle b\rangle$ and each $\langle b\rangle$ is a proper subgroup of $K$ (here we use that $K$ is noncyclic), $H=\bigcup_{\overset{b\in K}{b\neq 1}} \varphi^{-1}(\langle b\rangle)$ and each $\varphi^{-1}(\langle b\rangle)$  is a proper subgroup of $H$ (here we use that $\varphi$ is an epimorphism). Furthermore, by hypothesis, $f$ admits a local section on $\varphi^{-1}(\langle b\rangle)$ for each $b\in K$ with $b\neq 1$. This implies that $\mathrm{sec}(f)\leq |K|-1<\infty$ (the last inequality follows because $K$ is finite).

 Now, suppose that $H$ is abelian. Let $\mathrm{sec}(f)=m<\infty$ and consider $H_1\ldots,H_m$ proper subgroups of $H$ such that $H=H_1\cup\cdots\cup H_m$ and $f$ admits a local section on each $H_j$. Define $S=H_1\cap\cdots\cap H_m$ which is a normal subgroup of $H$ (here we use that $H$ is abelian). By \cite[Corollary 1.1, p. 2]{rao1992}, $H/S$ is a finite group. Observe that $H/S=\bigcup_{j=1}^{m} H_j/S$ and each $H_j/S$ is a proper subgroup of $H/S$ (here we use that $H_j$ is a proper subgroup of $H$ and $S\subseteq H_j$), hence $H/S$ is noncyclic. Consider $K=H/S$ and $\varphi:H\to K$ the usual quotient map (i.e., $\varphi(h)=h+S$ for any $h\in H$). Let $h+S\in H/S$ with $h\not\in S$. Observe that $h\in H_j$ for some $j\in\{1,\ldots,m\}$, then $\varphi^{-1}\left(\langle h+S\rangle\right)\subseteq H_j$ (here we use that $S\subseteq H_j$). Since, $f$ admits a local section on $H_j$, $f$ admits a local section on $\varphi^{-1}\left(\langle h+S\rangle\right)$.  
\end{proof}

Recall that if $f:G\to H$ admits a global section, then it admits a local section on each subgroup of $H$. Hence, Theorem~\ref{thm:finiteness-sec} recovers \cite[Theorem 2, p. 3]{rao1992}.

\begin{corollary}
    Let $A$ be an abelian group. Then $\sigma(A)<\infty$ if and only if there exists an epimorphism $\varphi:A\to K$ such that $K$ is a finite noncyclic group. 
\end{corollary}
\begin{proof}
    It follows from Theorem~\ref{thm:finiteness-sec} that it applies to the identity map $\mathrm{id}_G:G\to G$. In this case $\mathrm{sec}(\mathrm{id}_G)=\sigma(G)$.
\end{proof}

\subsection{Invariance}
From \cite[p. 34]{zapata2024}, two homomorphisms $f:G\to H$ and $f':G'\to H$ are said to be \textit{fibrewise   
equivalent} (or FE) if there exist homomorphisms $\psi:G\to G'$ and $\varphi:G'\to G$ such that $f'\circ\psi=f$ and $f\circ\varphi=f'$. This can be represented by the following commutative diagrams:
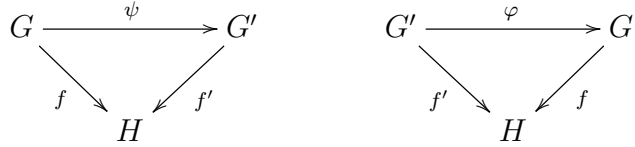
\begin{eqnarray*}
\xymatrix{ G \ar[rr]^{\psi} \ar[rd]_{f} & & G^{\prime} \ar[ld]^{f^\prime} & \\
        &  H & &} & \xymatrix{ G^{\prime} \ar[rd]_{f^\prime}  \ar[rr]^{\varphi}  & & G \ar[ld]^{f}& \\
        &  H & &}
\end{eqnarray*} Such homomorphisms $\psi$ and $\varphi$ are called \textit{fibrewise morphisms}.  

\medskip From \cite[Theorem 3.23]{zapata2024}, we obtain the FE-invariance of the sectional number.  

\begin{proposition}[Invariance]\label{prop:fe-invariance}
  Given homomorphisms $f:G\to H$ and $f^\prime:G^\prime\to H$, and a fibrewise morphism $\psi:G\to G^\prime$ (i.e., $f^\prime\circ \psi=f$). Then: \[\mathrm{sec}(f)\geq \mathrm{sec}(f').\] In particular, if $f$ and $f'$ are FE, then \[\mathrm{sec}(f)=\mathrm{sec}(f').\]
\end{proposition}

We have the following example.

\begin{example}
    Let $f:G\to H$ be a homomorphism.
    \begin{enumerate}
        \item[(1)] Let $K$ be a group. Consider the following commutative diagrams: 
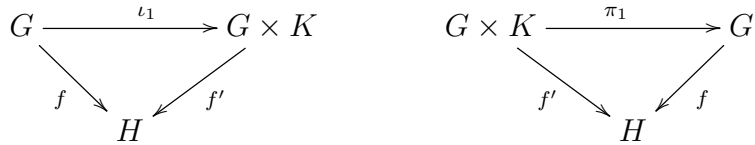
\begin{eqnarray*}
\xymatrix{ G \ar[rr]^{\iota_1} \ar[rd]_{f} & & G\times K \ar[ld]^{f^\prime} & \\
        &  H & &} & \xymatrix{ G\times K \ar[rd]_{f^\prime}  \ar[rr]^{\pi_1}  & & G \ar[ld]^{f}& \\
        &  H & &}
\end{eqnarray*} where $\pi_1$ and $\iota_1$ are the usual projection and injection, respectively. By Proposition~\ref{prop:fe-invariance}, we have  $\mathrm{sec}(f\circ\pi_1)=\mathrm{sec}(f)$.  

Likewise, we obtain that the equality $\mathrm{sec}(f\circ\pi_2)=\mathrm{sec}(f)$ holds also.
\item[(2)] Consider an abelian group $(G,+)$. Let $f_+:G\times G\to H$ given by $f_+(a,b)=f(a+b)$ for any $(a,b)\in G\times G$. Observe that $f_+$ is a homomorphism.

Now, consider the following commutative diagrams: 
\begin{eqnarray*}
\xymatrix{ G \ar[rr]^{\iota_1} \ar[rd]_{f} & & G\times G \ar[ld]^{f_+} & \\
        &  H & &} & \xymatrix{ G\times G \ar[rd]_{f_+}  \ar[rr]^{+}  & & G \ar[ld]^{f}& \\
        &  H & &}
\end{eqnarray*} where $\iota_1$ is the usual injection. By Proposition~\ref{prop:fe-invariance}, we have $\mathrm{sec}(f_+)=\mathrm{sec}(f)$. In particular, $\mathrm{sec}(+)=\sigma(G)$.
    \end{enumerate}
\end{example}

\subsection{Upper bound} Before to present an upper for the sectional number, we present the notion of cyclic covering number (cf. after of \cite[Example 3.12]{zapata2024}).

\begin{definition}[Cyclic Covering Number]\label{defn:cyclic-cov-number}
  Let $G$ be a group. The \textit{cyclic covering number} of $f$, denoted by $\sigma_c(G)$, is the least positive integer $m$ such that there exist cyclic proper subgroups $C_1,\ldots,C_m$ of $G$ such that $G=C_1\cup\cdots\cup C_m$. We set $\sigma_c(G)=\infty$ if no such $m$ exists. 
\end{definition}

Observe that $\sigma_c(G)\geq\sigma(G)\geq 3$. If $G$ is cyclic, then $\sigma_{c}(G)=\infty$. The other implication is not true; for example, we have $\sigma_{c}(\mathbb{Q})=\infty$ because $\sigma_{c}(\mathbb{Q})\geq \sigma(\mathbb{Q})=\infty$ and $\mathbb{Q}$ is not cyclic. Furthermore, any cyclic subgroup of a noncyclic group is proper. 

\medskip Let $G$ be a finite noncyclic group. Every (proper) cyclic subgroup $\langle x\rangle$ of $G$ has $\varphi(o(x))$ generators. For each $x\in G$, the number of distinct cyclic subgroups of order $o(x)$ is given by \[\dfrac{\text{the number of distinct elements of order $o(x)$}}{\varphi(o(x))},\] where $\varphi$ is Euler's totient function. Hence, the number of distinct nontrivial proper cyclic subgroups of $G$ is $\sum_{\overset{x\in G}{x\neq 1}}\dfrac{1}{\varphi(o(x))}$, and thus \begin{eqnarray}\label{eqn:ciclic-number}
    \sigma_c(G)&\leq\displaystyle\sum_{\overset{x\in G}{x\neq 1}}\dfrac{1}{\varphi(o(x))}.
\end{eqnarray}  

\medskip Inequality~(\ref{eqn:ciclic-number}) can be strict (see Example~\ref{exam:cyclic-z2z2}(3) below). We have the following example.

\begin{example}\label{exam:cyclic-z2z2}
    \noindent \begin{enumerate}
        \item[(1)] Let $G$ be a finite noncyclic \textit{elementary $p$-group}\footnote{Observe that any elementary $2$-group is abelian. Furthermore, for every prime number $p\geq 3$, the Heisenberg group $H(\mathbb{F}_p)$ is a non-abelian elementary $p$-group of order $p^3$.} for some prime $p$, i.e., all the non-identity elements of $G$ have the same order $p$. In this case, $\varphi(o(x))=\varphi(p)=p-1$ for any $x\in G$, $x\neq 1$. In addition, if $\langle x\rangle\subseteq \langle y\rangle$ for some $x,y\in G\setminus\{1\}$, then $\langle x\rangle= \langle y\rangle$. Hence, \[\sigma_c(G)= \dfrac{|G|-1}{p-1}.\] 
        \item[(2)] Let $n\geq 2$ be an integer and $p$ be a prime. Let $\mathbb{Z}_p^n=\mathbb{Z}_p\times\cdots\times \mathbb{Z}_p$ ($n$ times). We have \[\sigma_{c}\left(\mathbb{Z}_p^n\right)=\dfrac{p^n-1}{p-1}\qquad \text{ and }\qquad \sigma\left(\mathbb{Z}_p^n\right)=p+1.\] For instance, $\sigma_{c}\left(\mathbb{Z}_p\times\mathbb{Z}_p\right)=\sigma\left(\mathbb{Z}_p\times\mathbb{Z}_p\right)=p+1$.
        \item[(3)] Let $Q_8:=\{1,i,j,k,-1,-i,-j,-k\}$ be the quaternion group under multiplication (recall that $i^2=j^2=k^2=ijk=-1$). Note that its only cyclic subgroups are $\langle 1\rangle=\{1\}, \langle -1\rangle=\{1,-1\}, \langle i\rangle=\{1,i,-1,-i\}, \langle j\rangle=\{1,j,-1,-j\},$ and $\langle k\rangle=\{1,k,-1,-k\}$. Observe that $Q_8=\langle i\rangle\cup\langle j\rangle\cup\langle k\rangle$, and thus $\sigma_{c}\left(Q_8\right)=\sigma\left(Q_8\right)=3$.
    \end{enumerate} 
\end{example}

Now, we have the following upper bound for the sectional number (cf. \cite[Theorem 3.32]{zapata2024}).

\begin{theorem}[Upper Bound]\label{thm:upper-bound}
 Let $f:G\to H$ be a locally sectionable homomorphism. We have \[\text{sec}(f)\leq\sigma_c(H).\]     
\end{theorem}
\begin{proof}
     Since $f$ is locally sectionable, we have \[H=\bigcup_{\lambda\in\Lambda} H_\lambda,\] where $H_\lambda$ is a subgroup of $H$ provided of a local section $\sigma_\lambda:H_\lambda\to G$ of $f$, for each $\lambda\in\Lambda$. 
    
  If $\sigma_c(H)=\infty$, then the inequality $\text{sec}(f)\leq\sigma_c(H)$ always holds. Now, suppose $\sigma_c(H)=k<\infty$ and consider $\{C_1,\ldots,C_k\}$ a collection of cyclic proper subgroups of $H$ such that $H=C_1\cup\cdots\cup C_k$. Let $C_j=\langle b_j\rangle$ for each $j=1,\ldots,k$. Note that, for each $j=1,\ldots,k$, there exists $\lambda_j\in\Lambda$ such that $b_j\in H_{\lambda_j}$ and thus $C_j\subseteq H_{\lambda_j}$. Hence, the restriction $\sigma_{\lambda_j}:C_j\to G$ is a local section of $f$ for each $j=1,\ldots,k$. Therefore, $\text{sec}(f)\leq k=\sigma_c(H)$.
\end{proof}

Theorem~\ref{thm:upper-bound} together with the inequality $\mathrm{sec}(f)\geq \sigma(H)$ imply the following result.

\begin{corollary}\label{prop:cov-equal-cyclic}
 Let $f:G\to H$ be a locally sectionable homomorphism. Then \[\sigma(H)\leq\text{sec}(f)\leq\sigma_c(H).\] If $\sigma(H)=\sigma_c(H)$, then $\text{sec}(f)=\sigma_c(H)=\sigma(H)$.     
\end{corollary}

We have the following example.

\begin{example}\label{exam:proper-sec-number-codomain-z2z2}
  \noindent\begin{enumerate}
      \item[(1)] Let $G$ be an elementary $p$-group, i.e., any nontrivial element of $G$ has order $p$. Let $f:G\to \mathbb{Z}_p\times\mathbb{Z}_p$ be an epimorphism (and, of course, it is locally sectionable, see Lemma~\ref{lem:torsion-free-localsec-equiv-epi}(2)).  By Example~\ref{exam:cyclic-z2z2}(2), $\sigma_{c}(\mathbb{Z}_p\times\mathbb{Z}_p)=\sigma(\mathbb{Z}_p\times\mathbb{Z}_p)=p+1$. Hence, by Corollary~\ref{prop:cov-equal-cyclic}, $\mathrm{sec}(f)= p+1$. 
      \item[(2)]  Let $G$ be an elementary $3$-group. Let $f:G\to \mathbb{Z}_3\times\mathbb{Z}_3\times\mathbb{Z}_3$ be an epimorphism (and, of course, it is locally sectionable, see Lemma~\ref{lem:torsion-free-localsec-equiv-epi}(2)). Observe that $\sigma_{c}(\mathbb{Z}_3\times\mathbb{Z}_3\times\mathbb{Z}_3)=13$ and $\sigma(\mathbb{Z}_3\times\mathbb{Z}_3\times\mathbb{Z}_3)=4$ (see Example~\ref{exam:cyclic-z2z2}(2)). By Theorem~\ref{thm:upper-bound}, \[4\leq \mathrm{sec}(f)\leq 13.\] 
  \end{enumerate} 
\end{example}

\section{Applications}\label{sec:applications}
In this section, we present several applications motivated by the categorical and cohomological viewpoint. 
\subsection{Weak pullback} From \cite[Definition 2.5]{zapata2024}, a \textit{weak pullback} is a strictly commutative diagram of homomorphisms of the form:
\begin{eqnarray}\label{xfyzz}
\xymatrix{ \rule{3mm}{0mm}& G^\prime \ar[r]^{\varphi'} \ar[d]_{f^\prime} & G \ar[d]^{f} & \\ &
       H^\prime  \ar[r]_{\,\,\varphi} &  H &}
\end{eqnarray}
such that for any strictly commutative diagram as shown on the left side of~(\ref{ddiagramadoble}), there exists a (not necessarily unique) morphism $h:Z\to X^\prime$ that makes the diagram on the right side of~(\ref{ddiagramadoble}) commute:
\begin{eqnarray}\label{ddiagramadoble}
\xymatrix{
L \ar@/_10pt/[dr]_{\alpha} \ar@/^30pt/[rr]^{\beta} & & G \ar[d]^{f}  & & &
L\rule{-1mm}{0mm} \ar@/_10pt/[dr]_{\alpha} \ar@/^30pt/[rr]^{\beta}\ar[r]^{h} & 
G^\prime \ar[r]^{\varphi'} \ar[d]_{f^\prime} & G \\
& H^\prime  \ar[r]_{\,\,\varphi} &  H & & & & H^\prime &  \rule{3mm}{0mm}}
\end{eqnarray}

Note that if ~(\ref{xfyzz}) is a pullback (i.e., such $h$ is unique), then it is a weak pullback.  

\medskip The following example shows a weak pullback which is not a pullback.

\begin{example}\label{weak-no-pullback} Let $f:G\to H$ and $g:K\to H$ be homomorphisms. Consider the \textit{canonical pullback} $K\times_H G=\{(a,b)\in K\times H:~g(a)=f(b)\}$, it is also call the \textit{fiber-product}. Let $S$ be a group such that there exist at least two distinct homomorphism from $K\times_H G$ to $S$. Note that the following diagram:  
\begin{eqnarray*}
\xymatrix{  (K\times_H G)\times S \ar[rrr]^{\pi_{2}} \ar[d]_{\pi_1}&& & G \ar[d]^{f} & \\ 
       K  \ar[rrr]_{g}&& &  H &}
\end{eqnarray*} is a weak pullback however it is not a pullback. 
\end{example}

Recall that if there exists an epimorphism from $H'$ to $H$, then $\sigma(H')\leq\sigma(H)$. This statement can be generalized to the following statement that describes the behavior of the sectional number under weak pullbacks (cf. \cite[Theorem 3.25]{zapata2024}).

\begin{theorem}[Under Weak Pullbacks]\label{thm:bajo-weak-pullback}
   Assume that ~(\ref{xfyzz}) is a weak pullback. If $\varphi$ is an epimorphism or $f'$ does not admit a global section, then \[\mathrm{sec}(f')\leq\mathrm{sec}(f).\] 
\end{theorem}
\begin{proof}
Let $s:L\to G$ be a local section of $f$ defined in a subgroup $L$ of $H$. Consider the restriction homomorphism $\varphi_|:\varphi^{-1}(L)\to L$.
    
    Now, consider the following diagram given by the left side of~(\ref{1-diagramadobleee}). Since the diagram~(\ref{xfyzz}) is a weak pullback, there exists a homomorphism $\widetilde{s}:\varphi^{-1}(L)\to G'$ making commutative the diagram on the right side of~(\ref{1-diagramadobleee}).
\begin{eqnarray}\label{1-diagramadobleee}
\xymatrix{
\varphi^{-1}(L) \ar@/_10pt/[dr]_{\mathrm{incl}} \ar@/^30pt/[rr]^{s\circ \varphi_{|}} & & G \ar[d]^{f}  & & &
\varphi^{-1}(L)\rule{-1mm}{0mm} \ar@/_10pt/[dr]_{\mathrm{incl}} \ar@/^30pt/[rr]^{s\circ \varphi_{|}}\ar[r]^{\widetilde{s}} & 
G' \ar[r]^{\varphi'} \ar[d]_{f'} & G \\
& H'  \ar[r]_{\,\,\varphi} &  H & & & & H' &  \rule{3mm}{0mm}}
\end{eqnarray} Then, $\widetilde{s}$ is a local section of $f'$ defined in the subgroup $\varphi^{-1}(L)$ of $H'$.

Now, if $\mathrm{sec}(f)=k$ and $\{H_1,\ldots,H_k\}$ is a collection of proper subgroups of $H$ such that $H=H_1\cup\cdots\cup H_k$ and in each $H_i$ there exists a local section $s_i:H_i\to G$ of $f$. By the construction above, in each subgroup $\varphi^{-1}(H_i)$ there exists a local section $\widetilde{s}_i:\varphi^{-1}(H_i)\to G'$ of $f'$. Since $\varphi$ is an epimorphism or $f'$ does not admit a global section, $\{\varphi^{-1}(H_1),\ldots,\varphi^{-1}(H_k)\}$ is a collection of proper subgroups of $H'$ satisfying $H'=\varphi^{-1}(H_1)\cup\cdots\cup\varphi^{-1}(H_k)$. Hence, we conclude that $\mathrm{sec}(f')\leq k=\mathrm{sec}(f)$.  
\end{proof}

We have the following example.

\begin{example}\label{exam:sec-codomain-3}
      Let $K$ be a group such that $\sigma(K)=3$. By \cite[Theorem 2, p. 492]{harver1959}, there exists an epimorphism $\varphi:K\to \mathbb{Z}_2\times\mathbb{Z}_2$. Let $G$ be an (abelian) group such that any nontrivial element of $G$ has order two. Let $f:G\to \mathbb{Z}_2\times\mathbb{Z}_2$ be any epimorphism. By Example~\ref{exam:proper-sec-number-codomain-z2z2}, we have $\mathrm{sec}(f)=3$. Then, for any $f':G'\to K$ such that the commutative diagram:
  \begin{eqnarray*}
\xymatrix{ \rule{3mm}{0mm} G^\prime \ar[rr]^{\varphi'} \ar[d]_{f^\prime} & & G \ar[d]^{f}  \\ 
       K  \ar[rr]_{\,\,\varphi} &  & \mathbb{Z}_2\times\mathbb{Z}_2 }
\end{eqnarray*} is a weak pullback, we have $\mathrm{sec}(f')=3$ (see Theorem~\ref{thm:bajo-weak-pullback}). 
\end{example}

\subsection{Direct product} 
Let $f_1:G_1\to H_1$ and $f_2:G_2\to H_2$ be homomorphisms. The product map $f_1\times f_2:G_1\times G_2\to H_1\times H_2$ defined by:
\[\left(f_1\times f_2\right)(a_1,a_2)=\left(f_1(a_1),f_2(a_2)\right),~\forall (a_1,a_2)\in G_1\times G_2,\] is a homomorphism. In addition, given  $f:G\to H_1$ and $g:G\to H_2$ homomorphisms, the map $(f,g):G\to H_1\times H_2$ defined by:
\[(f,g)(a)=\left(f(a),g(a)\right),~\forall a\in G,\] is also a homomorphism.

\medskip  Proposition~\ref{prop:diagram} and Theorem~\ref{thm:bajo-weak-pullback} imply the following result.

 \begin{theorem}[Product]\label{thm:product}
  \noindent\begin{enumerate}
      \item[(1)] Let $f_1:G_1\to H_1$ and $f_2:G_2\to H_2$ be homomorphisms. If $f_1$ does not admit a global section, then \[\mathrm{sec}(f_1)\leq\mathrm{sec}(f_1\times f_2).\] Likewise, if $f_2$ does not admit a global section, then $\mathrm{sec}(f_2)\leq\mathrm{sec}(f_1\times f_2)$.
      \item[(2)] Let $f:G\to H$ be a homomorphism and $K$ be a group. We have \[\mathrm{sec}(f\times \mathrm{id}_K)\leq \mathrm{sec}(f)\] and the equality holds whenever $f$ does not admit a global section.  
  \end{enumerate}   
 \end{theorem}
 \begin{proof}
     \noindent\begin{enumerate}
         \item[(1)] It follows by applying the Proposition~\ref{prop:diagram} to the following commutative diagrams:
         \begin{eqnarray*}
\xymatrix{ G_1 \ar[d]^{f_1} & G_1\times G_2 \ar[r]^{\pi_1} \ar[d]_{f_1\times f_2} & G_1 \ar[d]^{f_1} & \\
     H_1 \ar[r]_{\iota_1} & H_1\times H_2  \ar[r]_{\pi_1} &  H_1 &} & \xymatrix{ G_2 \ar[d]^{f_2} & G_1\times G_2 \ar[r]^{\pi_2} \ar[d]_{f_1\times f_2} & G_2 \ar[d]^{f_2} & \\
     H_2 \ar[r]_{\iota_2} & H_1\times H_2  \ar[r]_{\pi_2} &  H_2 &}
\end{eqnarray*} where $\iota_1$ and $\iota_2$ are the usual inclusions, and $\pi_1$ and $\pi_2$ are the usual coordinate projections. 
\item[(2)] Note that the following diagram is commutative: 
\begin{eqnarray*}
\xymatrix{ \rule{3mm}{0mm}& G\times K \ar[rr]^{\pi_1} \ar[d]_{f\times\mathrm{id}_K}& & G \ar[d]^{f}  \\ 
&       H\times K  \ar[rr]_{\pi_1}& &  H }
\end{eqnarray*} and is a pullback. Since the usual first coordinate projection, $\pi_1$, is an epimorphism, by Theorem~\ref{thm:bajo-weak-pullback}, we conclude that the inequality $\mathrm{sec}(f\times \mathrm{id}_K)\leq\mathrm{sec}(f)$ always holds.

Now, suppose that $f$ does not admit a global section, by the Item (1), we have $\mathrm{sec}(f\times \mathrm{id}_K)\geq\mathrm{sec}(f)$. Therefore, $\mathrm{sec}(f\times \mathrm{id}_K)=\mathrm{sec}(f)$.
     \end{enumerate}
 \end{proof} 
\subsection{Lower bound}
Let $H$ be a group such that $\sigma(H)<\infty$. A \textit{categorical covering} of $H$ is a collection $\{H_i\}_{i=1}^{\ell}$ of proper subgroups of $H$ such that $H=\bigcup_{i=1}^{\ell} H_i$ and $\ell=\sigma(H)$. 

\medskip Observe that, given a homomorphism $f:G\to H$ with $\sigma(H)<\infty$, if for every categorical covering  $\{H_i\}_{i=1}^{\sigma(H)}$ of $H$, each $H_i$ admits a local section of $f$, then $\mathrm{sec}(f)=\sigma(H)$. On the other hand, we have the following lower bound.

\begin{theorem}[Lower Bound]\label{thm:lower-bound}
    Let $f:G\to H$ be a homomorphism such that $\sigma(H)<\infty$. If for every categorical covering  $\{H_i\}_{i=1}^{\sigma(H)}$ of $H$, at least one subgroup $H_i$ does not admit a local section of $f$, then \[\mathrm{sec}(f)>\sigma(H).\]
\end{theorem}
\begin{proof}
    Suppose $\mathrm{sec}(f)\leq\sigma(H)$. Let $H_1,\ldots, H_{\sigma(H)}$ be proper subgroups of $H$ such that $H=H_1\cup\ldots\cup H_{\sigma(H)}$, and $f$ admits a local section over each $H_i$. Observe that $\{H_i\}_{i=1}^{\sigma(H)}$ is a categorical covering of $H$. It contradicts the hypothesis of the theorem. Therefore, $\mathrm{sec}(f)>\sigma(H)$. 
\end{proof}

From the proof of \cite[Theorem 3, p. 3]{rao1992} we obtain the following statement.

\begin{proposition}
 \noindent\begin{enumerate}
     \item[(1)]  Let $f_1:G_1\to H_1$ and $f_2:G_2\to H_2$ be homomorphisms. If $H_1\times H_2$ is finite and abelian, then \[\min\{\mathrm{sec}(f_1),\mathrm{sec}(f_2)\}\leq\mathrm{sec}(f_1\times f_2)\] and the equality holds whenever $f_1\times f_2$ admits a global section (and, of course, $f_1$ and $f_2$ also admit global sections).
     \item[(2)] Let $f:G\to H$ be a homomorphism. Assume that $H$ is a finite abelian group and $H=A\oplus B$. Let $f_{| A}:f^{-1}(A)\to A$ and $f_{| B}:f^{-1}(B)\to B$ be the usual restrictions of $f$. Then \[\min\{\mathrm{sec}(f_{| A}),\mathrm{sec}(f_{| B})\}\leq\mathrm{sec}(f)\] and the equality holds whenever $f$ admits a global section (and, of course, $f_{| A}$ and $f_{| B}$ also admit global sections). 
 \end{enumerate}   
\end{proposition}
 
\subsection{$H$-points}
We present the following problem: Let $G$ and $H$ be groups. Given elements $a\in G$ and $b\in H$, when does there exist a homomorphism from $G$ into $H$ that takes $a$ into $b$?

\medskip Observe that, if $a$ is of finite order and $b$ is of infinite order, then there is no homomorphism from $G$ into $H$ that takes $a$ into $b$. In addition, in the case that $G=\langle a\rangle$ is a cyclic group, we have the following statement.

\begin{lemma}\label{prop:ciclico-existencia-hom}
 Let $H$ be a group and $G=\langle a\rangle$ be a cyclic group. Let $b\in H$, we have:
 \begin{enumerate}
     \item[(1)] If $a$ is of infinite order, then there exists $f\in\text{Hom}(\langle a\rangle,H)$ such that $f(a)=b$.
     \item[(2)] If $a$ and $b$ are of finite order. There exists $f\in\text{Hom}(\langle a\rangle,H)$ such that $f(a)=b$ if and only if $o(b)|o(a)$. 
 \end{enumerate}
\end{lemma}
\begin{proof}
    It follows from considering the map $f:\langle a\rangle\to H$ given by $f(a^j)=b^j$ for all $j\in\mathbb{Z}$.
\end{proof}

Motivated by the problem mentioned above and Lemma~\ref{prop:ciclico-existencia-hom}, we present the following notion.

\begin{definition}[$H$-point]\label{defn:h-punto}
  Let $G,H$ be groups and $a\in G$. We say that $a$ is a \textit{$H$-point} if for each element $b\in H$, there exists $f\in\mathrm{Hom}(G,H)$ such that $f(a)=b$.  
\end{definition}

We have the following example.

\begin{example}\label{exam:h-puntos}
 Let $G$ be a group. \begin{enumerate}
      \item[(1)] Any element $a\in G$ is a $0$-point because the trivial homomorphism $\overline{0}:G\to 0$ satisfies $\overline{0}(a)=0$ for any $a\in G$. 
      \item[(2)] The identity element $1\in G$ is not a $H$-point for any nontrivial group $H$.
      \item[(3)] Let $H$ be an abelian group and $a\in G$. Observe that $a$ is a $H$-point if and only if its inverse $a^{-1}$ is a $H$-point. 
      \item[(4)] Let $H$ be a group and suppose that $G=\langle a\rangle$ is an infinite cyclic group. By the proof of  Lemma~\ref{prop:ciclico-existencia-hom}(1), we have the (only) generators $a$ and $a^{-1}$ of $\langle a\rangle$ are $H$-points.   
  \end{enumerate}  
\end{example}

Let $H$ be an abelian group, $G$ be a group, and $a\in G$. The \textit{evaluation map} $\mathrm{ev}_a:\mathrm{Hom}(G,H)\to H,$ given by \[\mathrm{ev}_a(f)=f(a)\] for any $f\in \mathrm{Hom}(G,H)$, is a homomorphism. Here $\mathrm{Hom}(G,H)$ is equipped with the usual pointwise operation, i.e., $(fg)(x)=f(x)g(x)$ for any $x\in G$ and $f,g\in \mathrm{Hom}(G,H)$. Hence, $\mathrm{Hom}(G,H)$ is an abelian group.

\medskip Observe that $a$ is a $H$-point if and only if $\mathrm{ev}_a$ is surjective. 



\medskip We have the following example. It presents a characterization of $H$-points in terms of the sectional number.

\begin{example}\label{exam:k4-punto-numero-sec}
    Let $G$ be a group, and $a\in G$. Let $n\geq 2$ be an integer, observe that $\mathrm{Hom}(G,\mathbb{Z}_p^n)$ is an noncyclic elementary abelian $p$-group. Recall that $a$ is a $\mathbb{Z}_p^n$-point if and only if the evaluation map $\mathrm{ev}_a:\mathrm{Hom}(G,\mathbb{Z}_p^n)\to \mathbb{Z}_p^n$ is an epimorphim (in general it is not bijective, for example, in the case that $G$ is not cyclic), and it is equivalent to $\mathrm{sec}\left(\mathrm{ev}_a\right)=p+1$ (see Example~\ref{exam:epi-vector-spaces}(2)). 
    
    Therefore, we conclude that $a$ is a $\mathbb{Z}_p^n$-point if and only if $\mathrm{sec}\left(\mathrm{ev}_a\right)=p+1$. 
\end{example}

\subsection{The poset $\mathfrak{L}(f)$}
Let $f:G\to H$ be a homomorphism. Define the poset $\mathfrak{L}(f)$ whose elements are all proper subgroups of $H$ that admit a local section of $f$, and $L\leq K$ if $L\subseteq K$. Likewise, consider the poset $\mathfrak{L}(H)$ whose elements are all proper subgroups of $H$, and $L\leq K$ if $L\subseteq K$. Observe that $\mathfrak{L}(f)$ is a subposet of $\mathfrak{L}(H)$, and they coincides whenever $f$ admits a global section. 

\medskip Let $(P,\leq)$ be a poset. An element $m\in P$ is called a \textit{maximal element} if there is no $x\in P$ such that $m<x$. That is, if $m\leq x$ implies $x=m$. In a finite poset, maximal elements always exist. In a Hasse diagram, maximal elements are topmost nodes--with no edges going upward from them. A maximal element is an element that is not strictly less than any other element.  A \textit{maximum element} is an element that is the largest in the entire set. In any poset there is at most one maximum element. 

\medskip A \textit{chain} is a totally ordered subset $C\subseteq P$, that is, for any $x,y\in C$, $x\leq y$ or $y\leq x$. An \textit{infinite ascending chain} is a chain with infinitely many distinct elements $x_1<x_2<x_3<\cdots$. A poset $P$ is \textit{Noetherian} if every ascending chain $x_1\leq x_2\leq x_3\leq \cdots$ eventually stabilizes, i.e., there exists $N\geq 1$ such that $x_n=x_{n+1}=x_{n+2}=\cdots$ for all $n\geq N$. Alternatively, there is no infinite strictly increasing chain in $P$. Observe that if a poset $P$ is Noetherian then every non-empty subset of $P$ has at least one maximal element (by Zorn's Lemma). The other implication always holds. Given a Noetherian poset $P$, if $a\in P$, then there exists a maximal element $x$ of $P$ such that $a\leq x$. In fact, consider the nom-empty subset $\mathcal{A}_a=\{b\in P:~a\leq b\}$ and let $x\in \mathcal{A}_a$ be a maximal element of $\mathcal{A}_a$. If $z\in P$ and $x\leq z$, then $z\in \mathcal{A}_a$ (because $a\leq x$), and thus $z=x$ (because $x$ is a maximal element of $\mathcal{A}_a$). Therefore, $x$ is a maximal element of $P$ and $a\leq x$.  

\medskip We consider the following definition.

\begin{definition}[Covering Number of a Poset]\label{defn:cov-number-lp}
 Let $P:=P_X$ be a poset whose elements are subsets of a fixed set $X$. The \textit{covering number} of $P$, denoted by $\sigma(P)$, is the least positive integer $m$ such that there exist maximal elements $M_1,\ldots,M_m$ of $P$ such that $X=M_1\cup\cdots\cup M_m$. We set $\sigma(P)=\infty$ if no such $m$ exists. 
\end{definition}

\medskip We have the following result.

\begin{theorem}\label{thm:sec-poset}
  Let $f:G\to H$ be a homomorphism. We have \[\mathrm{sec}(f)\leq \sigma(\mathfrak{L}(f))\] and the equality holds whenever $\mathfrak{L}(f)$ is Noetherian (e.g., $H$ is finite or more general Noetherian). 
  \end{theorem}
\begin{proof}
  Inequality $\mathrm{sec}(f)\leq\sigma(\mathfrak{L}(f))$ follows from Definition~\ref{defn:cov-number-lp}. 

        Now assume that $\mathfrak{L}(f)$ is Noetherian. We will check the inequality $\sigma(\mathfrak{L}(f))\leq \mathrm{sec}(f)$. We can suppose $\mathrm{sec}(f)=m<\infty$; otherwise the inequality holds. Let $H_1,\ldots,H_m$ be proper subgroups of $H$ such that $H=H_1\cup\ldots\cup H_m$ and each $H_i$ admits a local section of $f$. The last condition implies that each $H_i$ is an element of $\mathfrak{L}(f)$. For each $i=1,\ldots,m$, let $M_i$ be a maximal element of $\mathfrak{L}(f)$ such that $H_i\leq M_i$ (here we use that $\mathfrak{L}(f)$ is Noetherian). Since $H=H_1\cup\ldots\cup H_m$, $H=M_1\cup\ldots\cup M_m$. Hence, $\sigma(\mathfrak{L}(f))\leq m=\mathrm{sec}(f)$. 
\end{proof}

Theorem~\ref{thm:sec-poset} implies the following example.

\begin{example}\label{exam:cov-cov-po}
    Let $G$ be a group. We have $\mathrm{sec}(\mathrm{id}_G)=\sigma(G)$ and $\mathfrak{L}(\mathrm{id}_G)=\mathfrak{L}(G)$. Hence, \[\sigma(G)\leq\sigma(\mathfrak{L}(G))\] and the equality holds whenever $G$ is Noetherian.
\end{example}

The equality $\sigma(G)=\sigma(\mathfrak{L}(G))$, given in Example~\ref{exam:cov-cov-po} for finite $G$, played a significant role in the proof of \cite[Theorem 3]{rao1992}.  
   
\subsection{Cohomology of groups}
In this section we show that $\mathrm{sec}(f)$ is the least number of subgroup restrictions needed to locally trivialize the 2-cocycle defining the extension. For this purpose, we follow the notation from \cite{brown2012}.  

\medskip Let $G$ be a group, and $A$ be a $G$-module. Let $\mathcal{E}(G,A)$ be the set of equivalence classes of extensions of $G$ by $A$ giving rise to the given action of $G$ on $A$. Then \begin{equation}\label{eqn:e-coho}
    \text{ there exists a bijection between $\mathcal{E}(G,A)$ and $H^2(G,A)$  }
\end{equation} (see \cite[Theorem 3.12, p. 93]{brown2012}). A homomorphism $\alpha:G'\to G$ induces a map $\alpha^\ast:\mathcal{E}(G,A)\to \mathcal{E}(G',A)$, which corresponds under the bijection~(\ref{eqn:e-coho}) to the homomorphism $H^2(\alpha,A):H^2(G,A)\to H^2(G',A)$. In fact, given $[w]\in\mathcal{E}(G,A)$ an equivalence class of the extension $0\to A\to E\to G\to 1$, $\alpha^\ast[w]=[w']$ where $w'$ is the extension given by $0\to A\to E\times_G G'\to G'\to 1$ with $E\times_G G'$ the fiber-product (see \cite[Exercise 1(a), p. 94]{brown2012}). In the case that $\alpha$ is the inclusion we write $\mathrm{res}^G_{G'}=H^2(\alpha,A)$ and call it a \textit{restriction map}. 

\medskip Let $[w]\in\mathcal{E}(G,A)$ be the equivalence class of the extension $0\to A\to E\to G\to 1$. We use the same notation $[w]\in H^2(G,A)$ for its corresponding $2$-cocycle via the bijection~(\ref{eqn:e-coho}). Recall that $[w]$ is the zero element in $H^2(G,A)$ if and only if its extension splits.  

\medskip Let $f:G\to H$ be an epimorphism such that $\mathrm{Ker}(f)$ is abelian. Observe that $G$ acts on $\mathrm{Ker}(f)$ by conjugation since $\mathrm{Ker}(f)$ is a normal subgroup of $G$; and the conjugation action of $\mathrm{Ker}(f)$ on itself is trivial because we suppose that $\mathrm{Ker}(f)$ is abelian, so there is an induced action of $H\cong G/\mathrm{Ker}(f)$ on $\mathrm{Ker}(f)$. We will consider $\mathrm{Ker}(f)$ as a $H$-module with this induced action. This is trivial if and only if $\mathrm{Ker}(f)\subseteq Z(G)$. 

\medskip We have the following result.

\begin{theorem}\label{thm:sec-coho}
    Let $f:G\to H$ be an epimorphism such that $\mathrm{Ker}(f)$ is abelian. Let $[w]$ be the equivalence class of the extension $0\to \mathrm{Ker}(f)\hookrightarrow G\stackrel{f}{\to} H\to 1$. We find that the sectional number $\mathrm{sec}(f)$ coincides with the least positive integer $m$ such that there exist proper subgroups $H_1,\ldots,H_m$ of $H$ such that $H=H_1\cup\cdots\cup H_m$, and $\mathrm{res}^H_{H_i}[w]=0$ in $H^2(H_i,\mathrm{Ker}(f))$ for each $i=1,\ldots,m$. 
\end{theorem}
\begin{proof}
 Given a subgroup $H'$ of $H$, observe that $\alpha^\ast[w]=[w']$ where $w'$ is the extension given by $0\to \mathrm{Ker}(f)\hookrightarrow f^{-1}(H')\stackrel{f_|}{\to} H'\to 1$. Hence, $\mathrm{res}^H_{H'}[w]=[w']$ in $H^2(H',\mathrm{Ker}(f))$. Therefore, $\mathrm{res}^H_{H'}[w]=0$ in $H^2(H',\mathrm{Ker}(f))$ if and only if $f_|:f^{-1}(H')\to H'$ admits a global section, which is equivalent to say that $f$ admits a local section over $H'$. 
\end{proof}

Theorem~\ref{thm:sec-coho} implies the following example.

\begin{example}
  Let $f:G\to H$ be an epimorphism such that $\mathrm{Ker}(f)$ is abelian. 
  \begin{enumerate}
      \item[(1)] Assume that $\mathrm{Ker}(f)$ and $H$ are finite. If there exist proper subgroups $H_1,\ldots,H_m$ of $H$ such that $H=H_1\cup\cdots\cup H_m$, and the order of each $H_i$ is coprime with the order of $\mathrm{Ker}(f)$ (and of course $H^2(H_i,\mathrm{Ker}(f))=0$, see \cite[Corollary 10.2, p. 84]{brown2012}), then \[\mathrm{sec}(f)\leq m.\]  
      \item[(2)] Assume that $\sigma(H)<\infty$. If for every categorical covering  $\{H_i\}_{i=1}^{\sigma(H)}$ of $H$, $\mathrm{res}^H_{H_i}[w]\neq 0$ in $H^2(H_i,\mathrm{Ker}(f))$ for at least one subgroup $H_i$, then \[\mathrm{sec}(f)>\sigma(H).\]
  \end{enumerate}
\end{example}


\section*{Acknowledgements}
The first author would like to thank grant\#2023/16525-7, and grant\#2022/16695-7, S\~{a}o Paulo Research Foundation (FAPESP) for financial support. The second author would like to thank the FC-UNI project funding. 

\section*{Conflict of Interest Statement}
On behalf of all authors, the corresponding author declares that there are no conflicts of interest.

\bibliographystyle{plain}

\end{document}